\documentclass[12pt]{amsproc}
\usepackage{xy, amssymb, hyperref,color}
\hypersetup{pdfauthor={Wesley Calvert, Kunal Dutta, and Amritanshu Prasad}, pdftitle={Degeneration and orbits of tuples and subgroups in abelian groups}}
\xyoption{all}

\theoremstyle{plain}
\newtheorem{theorem}[equation]{Theorem}
\newtheorem*{theorem*}{Theorem}

\newtheorem*{prop*}{Proposition}
\newtheorem{lemma}[equation]{Lemma}
\newtheorem*{lemma*}{Lemma}

\newtheorem*{cor*}{Corollary}

\newtheorem*{fq*}{Main question}

\theoremstyle{definition}
\newtheorem{defn}[equation]{Definition}
\newtheorem*{defn*}{Definition}
\newtheorem{remark}[equation]{Remark}
\newtheorem*{remark*}{Remark}

\newtheorem*{notation*}{Notation}

\newcommand{\N}{\mathbf N}

\newcommand{\tuple}[1]{\mathbf{#1}}

\newcommand{\Z}{\mathbf Z}

\DeclareMathOperator{\End}{End}

\title[Tuples in an Abelian group]{Degeneration and orbits of tuples and subgroups in an Abelian group}
\author{Wesley Calvert}
\address{Department of Mathematics\\ Mail Code 4408\\ 1245 Lincoln
  Drive\\ Southern Illinois University\\ Carbondale, Illinois 62901
  USA}
\email{wcalvert@siu.edu}
\author{Kunal Dutta}
\address{The Institute of Mathematical Sciences\\CIT campus Taramani\\Chennai 600 113 India}
\email{kdutta@imsc.res.in}
\author{Amritanshu Prasad}
\address{The Institute of Mathematical Sciences\\CIT campus Taramani\\Chennai 600 113 India}
\email{amri@imsc.res.in}
\thanks{The first author was partially supported by a Fulbright-Nehru
  Senior Research Scholarship.}
\subjclass[2010]{20K10, 20K27, 20K30}
\keywords{Degeneration, Ulm invariants, back-and-forth relations, automorphism orbits, tuples, subgroups, reduced torsion abelian groups}
\begin{document}
\maketitle
\begin{abstract}
  A tuple (or subgroup) in a group is said to degenerate to another if the latter is an endomorphic image of the former.
  In a countable reduced abelian group, it is shown that if tuples (or finite subgroups) degenerate to each other, then they lie in the same automorphism orbit.
  The proof is based on techniques that were developed by Kaplansky and Mackey in order to give an elegant proof of Ulm's theorem.
  Similar results hold for reduced countably generated torsion modules over principal ideal domains.
  It is shown that the depth and the description of atoms of the resulting poset of orbits of tuples depend only on the Ulm invariants of the module in question (and not on the underlying ring).
  A complete description of the poset of orbits of elements in terms of the Ulm invariants of the module is given.
  The relationship between this description of orbits and a very different-looking one obtained by Dutta and Prasad for torsion modules of bounded order is explained.
\end{abstract}
\section{Introduction}
\label{sec:intro}
Let $A$ be a reduced countable torsion abelian group with automorphism group $G$.
For each positive integer $n$, $G$ acts on $A^n$ by the diagonal action:
\begin{equation*}
	g\cdot(a_1,\dotsc,a_n)=(g(a_1),\dotsc,g(a_n)) \text{ for all } g\in G,\;a_1,\dotsc,a_n\in A.
\end{equation*}
We wish to understand the nature of the orbits of this action.
We call these the orbits of $n$-tuples in $A$.
We are also interested in the orbits of subgroups of $A$ under the action of $G$.

More generally, let $R$ be a principal ideal domain, and let $A$ be reduced countably generated torsion $R$-module.
Let $G$ be the group of $R$-module automorphisms of $A$.
We wish to understand the orbits of $n$-tuples and submodules in $A$.
Besides the case where $R$ is the ring of rational integers (which corresponds to abelian groups), the most interesting case is where $R$ is the ring of polynomials in one variable with coefficients in a field.
In this case, isomorphism classes of $R$-modules correspond to similarity classes of linear endomorphisms of a vector space of possibly infinite dimension.
Automorphism orbits of $n$-tuples in an $R$-module describe the possible relative positions of $n$-tuples of vectors in the vector space with respect to an operator in the corresponding similarity class.

There are two reasons for the choice of hypotheses on $A$.
The first is that they define a class of modules for which there is a nice structure theorem, namely Ulm's theorem \cite{Ulm1933} (which is discussed in more detail in Section~\ref{sec:Ulm} of this article).
The second is that this class includes the class of finitely generated torsion modules which (since it includes finite abelian groups and similarity classes of matrices), is already of great interest and, from the perspective of this article, is not any easier.

By the primary decomposition of torsion modules, the problem reduces to the case where $A$ is primary.
Since primary modules over a principal ideal domain are the same as torsion modules over discrete valuation rings, we assume that $R$ is a discrete valuation ring and that $A$ is a reduced countably generated torsion module.
The isomorphism classes of such modules are determined by their Ulm invariants (see Section~\ref{sec:Ulm}).
Moreover, the set of possible Ulm invariants does not depend on the discrete valuation ring $R$.
Thus there is a bijective correspondence between the set of isomorphism classes of reduced countably generated $R$-modules and $R'$-modules for any two discrete valuation rings $R$ and $R'$.
This suggests comparing the sets of orbits of $n$-tuples for a fixed set of Ulm invariants as $R$ varies.

The case where $A$ has bounded order (but is not necessarily countably
generated) and $n=1$ is studied in \cite{Dutta2011a}.
In particular, a combinatorial description of automorphism orbits in finite abelian groups $p$-groups is obtained, giving a better understanding of the classical results of Miller \cite{Miller1905} and Birkhoff \cite{Birkhoff1935}.
 The notion of degeneration is used to parametrize the set of orbits
 by the lattice of order ideals in a poset $P_f$ which is independent
 of $R$ (see Section~\ref{sec:orbits-of-elements} for more detail).
A formula is obtained for the cardinality of the orbit associated to each ideal in $P_f$.
It is found that the cardinality depends on $R$ only through the order of its residue field, and that this cardinality is a polynomial in $q$ which is easily calculated.

 In this article, some of the results of \cite{Dutta2011a} are extended to the case where $A$ is a reduced countably generated torsion $R$-module and $n>1$.
It is shown that degeneration does give rise to a partial order on the
set of orbits of $n$-tuples for all positive integers $n$.
A similar result is obtained for the set of automorphism orbits of
finitely generated submodules  (Theorem~\ref{theorem:partial-order}).
It is shown that this poset of $G$-orbits in $A$ is independent of $R$ (Theorem~\ref{theorem:poset-elements}) and that the poset is not dependent on the values of the Ulm invariants, but only on the set of ordinals for which they are positive (the latter is the analog of the result in \cite{Dutta2011a} that the poset does not depend on the multiplicities of the cyclic factors).
 The proofs in \cite{Dutta2011a} relied on the fact that every finitely generated torsion $R$-module is a direct sum of cyclic modules.
 This fails in the more general case considered here.
 Moreover, it is not clear how to extend those methods to the analysis of orbits of $n$-tuples for $n>1$.
 Instead, we use ideas that come out of a proof of Ulm's theorem given by Kaplansky and Mackey \cite{Kaplansky1970,Kaplansky-Mackey}; mainly the fact that height-preserving isomorphisms of finite subgroups can be extended to isomorphisms of groups with the same Ulm invariants (Lemma~\ref{lemma:crucial}) and the resulting criterion due to Barker \cite[Proposition~3.2]{Barker1995} for a pair of $n$-tuples to lie in the same orbit.

As to whether or not there is an order-preserving correspondence between
orbits of $n$-tuples as $R$ varies over discrete valuation rings with the same residue field is a question that remains
unresolved for $n>1$. Theorems~\ref{theorem:depth} and
\ref{theorem:atoms} are some partial results in this direction; they
state that the depth of the poset of $n$-tuples, and the set of its
atoms, do not depend on $R$.

The Kaplansky-Mackey machinery can be used to describe the
poset of $G$-orbits of elements of $A$ by using Ulm sequences.
We discuss this in Section~\ref{sec:orbits-of-elements}. The
dictionary for going between Ulm sequences and the very different
combinatorial description of orbits in \cite{Dutta2011a} is given in Section~\ref{sec:ulm-sequences-order}.
\section{Ulm and Barker Theorems}
\label{sec:Ulm}
We now give a brief overview of Kaplansky and Mackey's proof of Ulm's theorem following \cite{Kaplansky1970}.

Let $R$ be a discrete valuation ring with prime ideal generated by $p$.
Let $A$ be a reduced torsion $R$-module.
For each ordinal $\alpha$, define a submodule $A_\alpha$ inductively as follows:
let $A_0=A$.
If $\alpha=\beta+1$ define $A_\alpha$ to be $pA_\beta$.
If $\alpha$ is a limit ordinal, define $A_\alpha$ to be the intersection of $A_\beta$ as $\beta$ ranges over all ordinals $\beta<\alpha$.

Let $P(A)$ denote the submodule of $A$ consisting of elements that are annihilated by $p$:
\begin{equation*}
	P(A)=\{a\in A\;|\;pa=0\}.
\end{equation*}
For each ordinal $\alpha$, let $P(A)_\alpha=P(A)\cap A_\alpha$.
The quotient $P(A)_\alpha/P(A)_{\alpha+1}$ may be regarded as a vector space over the residue field of $R$.
Its dimension is denoted by $f_\alpha(A)$ and called the $\alpha$th Ulm invariant of $A$.
\begin{theorem*}[Ulm]
	Two reduced countably generated torsion $R$-modules are isomorphic if and only if they have the same Ulm invariants.
\end{theorem*}

The notion of height is crucial to the proof of Ulm's theorem due to Kaplansky and Mackey: we say that a non-zero element $a\in A$ has height $\alpha$, and write $h(a)=\alpha$, if $a$ is in $A_\alpha$ but not in $A_{\alpha+1}$.
We set $h(0)=\infty$, with the symbol $\infty$ deemed to be greater than any ordinal.
The following lemma \cite[p. 30]{Kaplansky1970} is a crucial step in their proof:
\begin{lemma}
	\label{lemma:crucial}
	Let $R$ be a discrete valuation ring.
	Let $A$ and $B$ be reduced countably generated torsion $R$-modules with the same Ulm invariants.
	Suppose that $S$ and $T$ are finitely generated submodules of $A$ and $B$ and $V:S\to T$ is a height-preserving isomorphism.
	Let $a$ be any element of $A$ not in $S$.
	Then $V$ extends to a height-preserving isomorphism of the submodule generated by $S$ and $a$ to a submodule of $B$.
\end{lemma}
Kaplansky and Mackey prove Ulm's theorem from this lemma as follows: enumerate generating sets of $A$ and $B$ by sequences $\{a_n\}$ and $\{b_n\}$.
Starting with $S$ as the trivial submodule of $A$ and $T$ as the trivial submodule of $B$, use Lemma~\ref{lemma:crucial} to get a height-preserving isomorphism $V_1$ of the sub-module $S_1$ generated by $a_1$ onto a submodule $T_1$ of $B$.
Let $V_2$ be an extension of the inverse of $V_1$ to a height-preserving isomorphism of the submodule $T_2$ generated by $T_1$ and $b_1$ onto a submodule of $A$.
Continue going back and forth in this manner, taking care of the $n$th generator of $A$ at the $(2n-1)$st step and the $n$th generator of $B$ at the $2n$th step to eventually obtain an isomorphism of $A$ onto $B$.

Besides Ulm's theorem, an important corollary \cite[Exercise~38]{Kaplansky1970} of Lemma~\ref{lemma:crucial}, which is non-trivial even in the case of finitely generated torsion $R$-modules, is
\begin{lemma}[Extendability Lemma]
	\label{lemma:extensibility}
	Let $R$ be a discrete valuation ring.
	Let $A$ and $B$ be reduced countably generated torsion $R$-modules with the same Ulm invariants.
	Then every height-preserving isomorphism of a finitely generated submodule of $A$ onto a submodule of $B$ extends to an isomorphism of $A$ onto $B$.
\end{lemma}
	 Barker's criterion for two $n$-tuples to lie in the same orbit follows from Lemma~\ref{lemma:extensibility}:
\begin{theorem*}[Barker's criterion]
	Let $R$ be a principal ideal domain.
	Let $A$ be a reduced countably generated torsion $R$-module.
	The $n$-tuples $(a_1,\dotsc,a_n)$ and $(b_1,\dotsc,b_n)$ of $A$ lie in the same automorphism orbit if and only if
	\begin{equation}
		\label{eq:Barker}
		h(r_1a_1+\dotsb+r_na_n)=h(r_1b_1+\dotsb+r_nb_n) \text{ for all } r_1,\dotsc, r_n\in R.
	\end{equation}
\end{theorem*}
\begin{proof}
	The condition (\ref{eq:Barker}), because of our convention for the height of $0$, ensures that if $r_1a_1+\dotsc+r_na_n=0$, then $r_1b_1+\dotsc+r_nb_n=0$.
	Therefore, $a_i\mapsto b_i$ defines a height-preserving isomorphism from the submodule generated by $a_1,\dotsc,a_n$ onto the submodule generated by $b_1,\dotsc,b_n$.
	The extensibility lemma (Lemma~\ref{lemma:extensibility}) applied to this isomorphism ensures the existence of an automorphism of $A$ which takes $a_i$ to $b_i$ for each $i$.
\end{proof}

We need a lemma \cite[Exercise~39]{Kaplansky1970} about extending homomorphisms, whose proof is simpler than that of the Extendability Lemma.
\begin{lemma}
	\label{lemma:hom-extensibility}
	Let $R$ be a discrete valuation ring and $A$ and $B$ be reduced torsion $R$-modules, with $A$ countably generated.
	Let $S$ be a finitely generated submodule of $A$.
	Let $V$ be a height-increasing homomorphism of $S$ into $B$.
	Then $V$ can be extended to a homomorphism from all of $A$ into $B$.
\end{lemma}

\section{The Degeneration Partial Order}
\label{sec:degeneration}
\begin{defn*}[Degeneration]
	Let $A$ and $B$ be $R$-modules for some ring $R$.
	For $n$-tuples $\tuple a=(a_1,\dotsc,a_n)$ and $\tuple b=(b_1,\dotsc,b_n)$ in $A$ and $B$ respectively,
	we say that $\tuple a$ degenerates to $\tuple b$, and write $\tuple a\to \tuple b$, if there exists an $R$-module homomorphism $\phi:A\to B$ such that $\phi(a_i)=b_i$ for $i=1,\dotsc,n$.
	For submodules $S$ and $T$ in $A$ and $B$ respectively, we say that $S$ degenerates to $T$ (written as $S\to T$) if a homomorphism maps $S$ onto $T$.
\end{defn*}
We are primarily interested in the case where $A=B$.
Clearly, if $n$-tuples or subgroups of $A$ lie in the same $G$-orbit, then they degenerate to each other.
Barker's criterion implies the converse:
\begin{theorem}
	\label{theorem:partial-order}
	Let $R$ be a principal ideal domain and $A$ be a reduced countably generated torsion $R$-module.
	\begin{enumerate}
	\item \label{item:tuples} If $n$-tuples $\tuple a$ and $\tuple b$ of $A$ degenerate to each other ($\tuple a \to \tuple b$ and $\tuple b\to \tuple a$) then they lie in the same $G$-orbit.
	\item \label{item:subgroups} If finite submodules $S$ and $T$ degenerate to each other ($S\to T$ and $T\to S$) then they lie in the same $G$-orbit.
	\end{enumerate}
\end{theorem}
\begin{remark}
	The finiteness hypothesis on the submodules in \ref{item:subgroups} is satisfied for all finitely generated submodules if, for every maximal ideal $M$ of $R$, $R/M$ is a finite field.
	This includes finite abelian groups and polynomial rings over finite fields.
\end{remark}
\begin{proof}[Proof of \ref{item:tuples}]
	Suppose $\phi\in \End_RA$ is such that $\phi(a_i)=b_i$ for each $i$.
	Since homomorphisms are height-increasing, for all $r_1,\dotsc,r_n\in R$,
	\begin{equation*}
		h(r_1a_1+\dotsb+r_na_n)\geq h(r_1b_1+\dotsb+r_nb_n).
	\end{equation*}
	Thus if $\tuple a\to \tuple b$ and $\tuple b\to \tuple a$, we get
	\begin{equation*}
		h(r_1a_1+\dotsb+r_na_n)=h(r_1b_1+\dotsb+r_nb_n)
	\end{equation*}
	for all $r_1,\dotsc,r_n\in R$. 
	By Barker's criterion, $\tuple a$ and $\tuple b$ lie in the same $G$-orbit.
\end{proof}
\begin{proof}[Proof of \ref{item:subgroups}]
	Suppose than an endomorphism $\phi$ of $A$ maps $S$ onto $T$ and another endomorphism $\psi$ maps $T$ onto $S$.
	Since the restriction of $\psi\circ\phi$ to $S$ is a surjective endomorphism of the finite $R$-module $S$, must be invertible on $S$.
	Since the automorphism group of $S$ is finite, there exists a natural number $N$ such that $(\psi\circ\phi)^N$ is the identity on $S$.
	By replacing $\phi$ by $(\psi\circ\phi)^{N-1}\circ\psi$, one may assume that $\phi$ and $\psi$ are mutual inverses when restricted to $S$ and $T$ respectively.
	
	Let $a_1,\dotsc,a_n$ be a generating set for $S$, and set $b_i=\phi(a_i)$ for each $i$.
	These are $n$-tuples which degenerate to each other, so by \ref{item:tuples} there exists an automorphism of $A$ mapping one to the other.
	This automorphism also maps $S$ onto $T$.
\end{proof} 
Theorem~\ref{theorem:partial-order} says that degeneration induces a partial order on the set of $G$-orbits of $n$-tuples and subgroups of $A$.
\begin{theorem}[Criterion for degeneration]
	\label{theorem:partial-order-from-heights}
	Let $R$ be a discrete valuation ring and $A$ and $B$ be reduced torsion $R$-modules, with $A$ countably generated.
	For $n$-tuples $\tuple a$ and $\tuple b$ in $A$ and $B$ respectively, $\tuple a$ degenerates to $\tuple b$ if and only if
	\begin{equation}
		\label{eq:height-increasing}
		h(r_1a_1+\dotsb+r_na_n)\leq h(r_1b_1+\dotsb+r_nb_n)
	\end{equation}
	for all $r_1,\dotsc,r_n\in R$.
\end{theorem}
\begin{proof}
	As in the proof of Barker's theorem, the condition on heights ensures that $a_i\mapsto b_i$ gives rise to a well-defined height-increasing homomorphism from the submodule of $A$ generated by $a_1,\dotsc,a_n$ to $B$.
	By Lemma~\ref{lemma:hom-extensibility}, this homomorphism can be extended to a homomorphism on all of $A$ that takes $\tuple a$ to $\tuple b$.
\end{proof}
If the first part of Theorem~\ref{theorem:partial-order} tells us that degeneration descends to a partial order on the set of orbits of $n$-tuples in $A$,
Lemma~\ref{theorem:partial-order-from-heights} (with $A=B$) lets us read off the poset structure on the set of $G$-orbits in $A$ from heights of linear combinations.

\section{Height sequences and lengths of chains}
\label{sec:height-lattice}
For an $n$-tuple in $A$ and an ordinal $h$ define subgroups of $R^n$ by
\begin{equation*}
  M_h(\tuple a) = \{(r_1,\dotsc,r_n)\in R^n\;|\;h(r_1a_1+\dotsb+r_na_n)\geq h\}.
\end{equation*}
Thus each $n$-tuple $\tuple a$ gives rise to a chain of subgroups:
\begin{equation*}
  R^n = M_0(\tuple a) \supset M_1(\tuple a) \supset M_2(\tuple a)
  \supset \dotsb
\end{equation*}
which we call the \emph{height sequence of $\tuple a$}.

Theorem~\ref{theorem:partial-order} can be reformulated in terms of
height sequences:
\begin{theorem}
  \label{theorem:height-sequences}
  For $n$-tuples $\tuple a$ and $\tuple b$ in $A$, $\tuple a\to \tuple b$ if and only if $M_h(\tuple a)\subset M_h(\tuple b)$ for every ordinal $h$.
\end{theorem}
Although we are not yet in a position to describe the poset structure
of $G$-orbits of tuples under degeneration, the above reformulation of
Theorem~\ref{theorem:partial-order} allows us to give a bound on the
length of chains in the bounded order case which depends only on the
annihilator of $A$ and $n$ (and not on $R$):
\begin{theorem}
  \label{theorem:depth}
  Suppose that $A$ is annihilated by $p^k$.
  Then if
  \begin{equation*}
    \tuple a^{(0)}\to \tuple a^{(1)} \to \dotsb\to\tuple a^{(m)}
  \end{equation*}
  is a chain of degenerations with $m> nk(k+1)$ then there exists
  $0<i\leq m$ such that $\tuple a^{(i)}\to \tuple a^{(i-1)}$.
\end{theorem}
\begin{proof}
  Since $A$ is annihilated by $p^k$, $M_h(a)\supset p^kR^n$.
  For each $\tuple a\in A$, define
  \begin{equation}
    \label{eq:1}
    N(\tuple a) = \sum_{h=0}^k \log_p|M_h(\tuple a)/p^kR^n|.
  \end{equation}
  $N$ takes integer values between
  $0$ and $nk(k+1)$.
  Moreover, since $0$ is the only element of height $\geq k$, the
  invariants $M_0,\dotsc,M_k$ are enough to distinguish between orbits
  of $n$-tuples.

  Clearly, if $\tuple a\to \tuple b$ then $N(\tuple a)\leq N(\tuple
  b)$.
  More importantly, if $\tuple a\to \tuple b$ and $N(\tuple a)= N(\tuple
  b)$, then $M_h(a)=M_h(b)$ for all $h$, and so $\tuple b\to \tuple
  a$.
  If $m$ in the sequence (\ref{eq:1}) exceeds $nk(k+1)$, then
  $N(\tuple a^{(i-1)})=N(\tuple a^{(i)})$ for some $i$, an so $\tuple
  a^{(i)}\to \tuple a^{(i-1)}$.
\end{proof}

\section{Atoms}
\label{sec:atoms}

By an atom in the poset of orbits of $n$-tuples in $A$ under degeneration, we mean
an orbit which can only degenerate to itself or to $0$.
It turns out that the characterization of atoms for tuples can be
reduced to the characterization of atoms for the poset of orbits of
elements of $A$.
In the next section, we will see that the poset of orbits of elements
of $A$ is independent of $R$ (and in fact depends only on which Ulm
invariants are positive).
Thus the following characterization of atoms will show that the set of
atoms in the poset of orbits of $n$-tuples is independent of $R$.

\begin{theorem}
  \label{theorem:atoms}
  Let $A$ be a reduced countably generated torsion $R$-module.
  Suppose $\tuple a = (a_1,\dotsc,a_n)$ can only degenerate to itself or
  to $0$, and that $a_1\neq 0$.
  Then then $\tuple a$ is of the form
  \begin{equation*}
    (a_1,r_2a_1,\dotsc,r_na_1)
  \end{equation*}
  where $a$ is an element of $A$ which can only degenerate to itself
  or to $0$, and $r_2,\dotsc,r_n$ are arbitrary elements of $R$.
\end{theorem}
\begin{proof}
  The hypothesis obviously implies that $a$ can not degenerate to
  anything but itself or to $0$.
  It also implies that every element of $\tuple a$ is annihilated by
  $p$, since $x\mapsto px$ would lead to a non-trivial degeneration.
  Thus, if $\tuple a$ is not of the form claimed in the theorem, we
  may assume that some $a_i$ is a non-zero element of $P(A)$ --- that
  is, the submodule of $A$ consisting of elements annihilated by $p$ --- which is linearly
  independent of $a_1$ in the $R/(p)$-vector space $P(A)$.

  Let $S$ be the submodule of $A$ generated by $a_1$ and $a_i$.
  Since $a_1$ and $a_i$ are linearly independent, there exists an
  endomorphism of $S$ such that $a_1\mapsto a_1$ and $a_i\mapsto 0$
  (which is obviously height-increasing).
  Extending this (by Lemma~\ref{lemma:hom-extensibility}) to an
  endomorphism of $A$ gives a degeneration of $\tuple a$ to a non-zero
  tuple different from $\tuple a$, which is a contradiction.
\end{proof}

\section{Orbits of elements}
\label{sec:orbits-of-elements}
Kaplansky introduces the Ulm sequence of an element in \cite[Section~18]{Kaplansky1970}:
\begin{defn}[Ulm Sequence]
	Let $R$ be a discrete valuation ring with maximal ideal generated by $p$.
	For each $a\in A$, its height sequence is defined as the sequence $\{h_i\}_{i\geq 0}$ of ordinals where
	\begin{equation*}
		h_i=h(p^ia).
	\end{equation*}
	If $a\in A$ is a torsion element, then this sequence is strictly increasing until it stabilizes at the value $\infty$ after a finite number of steps.
\end{defn}
Ulm sequences characterize orbits and degenerations in reduced countably generated torsion $R$-modules (this is easily deduced from Barker's criterion; also see \cite[Theorem~24]{Kaplansky1970}):
\begin{theorem}
	\label{theorem:poset-elements}
	Let $R$ be a discrete valuation ring, and let $A$ be a reduced countably generated torsion $R$-module.
	Then 
	\begin{enumerate}
	\item $a$ degenerates to $b$ if and only if the Ulm sequence of $b$ dominates the Ulm sequence of $a$ term-wise.
	\item Two elements $a$ and $b$ of $A$ lie in the same $G$-orbit if and only if their Ulm sequences coincide.
	\end{enumerate}
\end{theorem}
Given the Ulm invariants of a reduced countably generated torsion $R$-module $A$, the set of sequences of ordinals which arise as Ulm sequences in $A$ are characterized by Kaplansky \cite[Lammas~22 and~24]{Kaplansky1970}:
\begin{theorem}
	\label{theorem:Ulm-sequence-characterization}
	Let $R$ be a discrete valuation ring and let $A$ be a reduced countably generated torsion $R$-module.
	A sequence $\{h_i\}$ of ordinals which is strictly increasing until it stabilizes at $\infty$ after a finite number of steps is the Ulm sequence of an element of $a$ if and only if, whenever $h_n>h_{n-1}+1$, the Ulm invariant $f_{h_{n-1}}(A)$ is positive.
\end{theorem}
Theorem~\ref{theorem:Ulm-sequence-characterization} allows us to give a combinatorial description of the poset of automorphism orbits in a reduced countably generated torsion module in terms of its Ulm invariants:
\begin{theorem}\label{theorem:poset-description}
	Let $R$ be a discrete valuation ring and let $A$ be a reduced countably generated torsion $R$-module with Ulm invariants $f=\{f_\alpha\}$.
	Let $H_f$ denote the set consisting of sequences of the form
	\begin{equation*}
		\{h_n\}=h_0,h_1,\dotsc,h_k,\infty,\infty,\dotsc
	\end{equation*}
	where $h_0<h_1<\dotsb<h_k$ are ordinals such that $f_{h_k}>0$, and
	whenever $h_n>h_{n-1}+1$, then $f_{h_{n-1}}>0$.
	For $\{h_n\}$ and $\{h'_n\}$ in $H_f$ say that $\{h_n\}\leq\{h'_n\}$ if and only if $h_n\geq h'_n$ for all $n\geq 0$.
	Then taking an element of $A$ to its Ulm sequence gives rise to an isomorphism of the poset of automorphism orbits in $A$ to $H_f$.	
\end{theorem}
Theorem~\ref{theorem:poset-description} allows us to compare the
posets of automorphism orbits for different modules with the same Ulm invariants, even while allowing the discrete valuation ring to vary:
\begin{theorem}[Poset comparison]
\label{theorem:poset-comparison}
	Suppose $R_1$ and $R_2$ are two discrete valuation rings, and $A_1$ and $A_2$ are reduced countably generated torsion modules over $R_1$ and $R_2$ respectively such that, for each ordinal $\alpha$, the Ulm invariant $f_\alpha(A_1)$ is positive if and only if the Ulm invariant $f_\alpha(A_2)$ is positive.
	Then the posets of automorphism orbits of elements in $A_1$ and $A_2$ are canonically isomorphic.
\end{theorem}

\section{Ulm Sequences and Order Ideals}
\label{sec:ulm-sequences-order}
The cyclic $R$-module $R/p^kR$ has $k$ orbits of non-zero elements under the action of its automorphism group, represented by $1,p,\ldots,p^{k-1}$.
Let $P$ be the disjoint union over all $k\in \N$ of the orbits of non-zero elements in $R/p^kR$.
$P$ can be represented by the set
\begin{equation*}
  P = \big\{(v,\alpha)|\alpha\in \N,\; 0\leq v<\alpha\big\},
\end{equation*}
where $(v,\alpha)$ denotes the orbit of $p^v$ in $R/p^\alpha R$.
Degeneracy descends to a relation on $P$ given by 
$(v,\alpha)\geq (v',\alpha')$ if and only if $v'\geq v$ and $\alpha'-v'\leq \alpha-v$ \cite[Lemma~3.1]{Dutta2011a}.
It follows that this reflexive and transitive relation is in fact a partial order.
The Hasse diagram for $P$, which we call the fundamental poset, is given in Figure~\ref{fig:funda}.
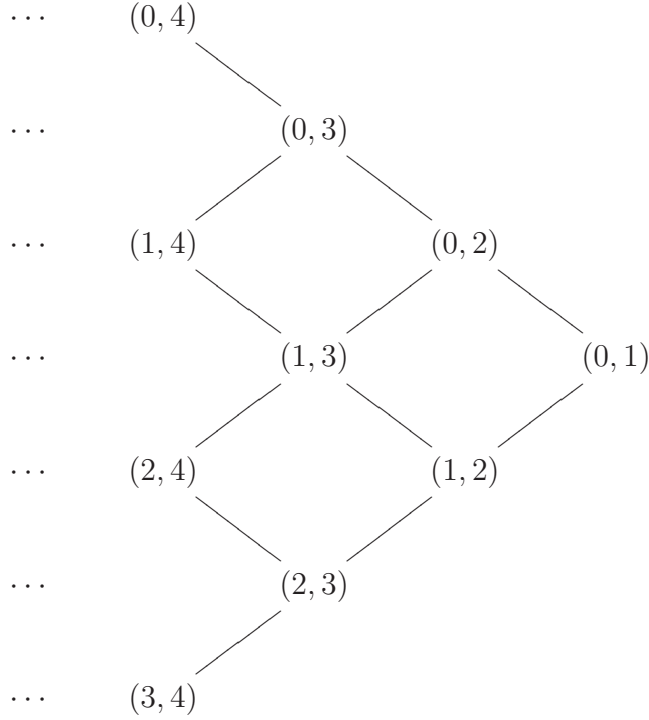
\begin{figure}
  \centering
  \begin{equation*}
    \xymatrix{
      \cdots & (0,4) \ar@{-}[dr] & & & \\
      \cdots && (0,3) \ar@{-}[dr] & & \\
      \cdots &(1,4) \ar@{-}[dr] \ar@{-}[ur] & & (0,2) \ar@{-}[dr] & \\
      \cdots && (1,3) \ar@{-}[dr] \ar@{-}[ur] & & (0,1)\\
      \cdots &(2,4) \ar@{-}[dr] \ar@{-}[ur] & & (1,2) \ar@{-}[ur] & \\
      \cdots && (2,3) \ar@{-}[ur] & & \\
      \cdots &(3,4) \ar@{-}[ur] & & &\\
    }
  \end{equation*}
  \caption{The fundamental poset $P$}
  \label{fig:funda}
\end{figure}

Now assume that $R$ is a discrete valuation ring and that $A$ is a torsion $R$-module of bounded order.
Then $A$ is a product of cyclic modules:
\begin{equation}
	\label{eq:bounded-order}
	A=\prod_{\alpha=1}^\infty (R/p^\alpha R)^{f_\alpha}
\end{equation}
for some sequence of cardinals $f=\{f_\alpha\}$.
Given $a\in A$, 
the order ideal in $P$ generated by the orbits of the non-zero coordinates $a_{\alpha i}\in R/p^\alpha R$ of $a$ with respect to the decomposition (\ref{eq:bounded-order}) is called the ideal of $a$, and denoted by $I(a)$. 

One may read off degeneracy of elements of $R$-modules of bounded order from the fundamental poset $P$.
\begin{theorem}[{\cite[Theorem~4.1]{Dutta2011a}}]
  \label{theorem:degeneracy}
  Let $R$ be a discrete valuation ring.
  Let $A$ and $B$ be torsion $R$-modules of bounded order.
  Given $a\in A$ and $b\in B$, $a$ degenerates to $b$ if and only if $I(a)\supset I(b)$.
  When $A=B$, then $a$ and $b$ lie in the same automorphism orbit if and only if $I(a)=I(b)$.
\end{theorem}

It also gives a concrete interpretation of the poset of orbits in a torsion $R$-module of bounded order under the degeneration partial order:
\begin{theorem}[{\cite[Theorem~5.4]{Dutta2011a}}]\label{theorem:old-poset-description}
	Let $R$ be a discrete valuation ring an let $A$ be a torsion $R$-module of bounded order as in (\ref{eq:bounded-order}).
	Let $P_f$ be the induced subposet $P$ consisting those $(v,\alpha)\in P$ for which $f_\alpha>0$. 
	The poset of $G$-orbits in $A$ is isomorphic to the lattice $J(P_f)$ of order ideals in $P_f$.
\end{theorem}
Therefore, for a countably generated torsion module $A$ of bounded order, there are two different parametrizations of the poset of automorphism orbits: $H_f$ from Theorem~\ref{theorem:poset-description},  and $J(P_f)$ from Theorem~\ref{theorem:old-poset-description}.
It follows that there is a canonical isomorphism between these posets, which we shall now make explicit.

Every countably generated torsion $R$-module of
bounded order is of the form
\begin{equation*}
  A = \bigoplus_{\alpha=1}^\infty (R/p^\alpha)^{f_\alpha}
\end{equation*}
for some sequence $f=\{f_\alpha\}_{\alpha=1}^\infty$ of non-negative integers, of
which only finitely many are strictly positive.
The integer $f_\alpha$ is the $\alpha$th Ulm invariant of $A$.

Given $\{h_0,h_1,\dotsc\}\in H_f$, let $I(\{h_n\})$ denote the order ideal of $P_f$ generated by
\begin{equation*}
	\{(h_{i-1}-i+1,h_{i-1}+1)\:|\: h_i>h_{i-1}+1\}.
\end{equation*}
Conversely, given an ideal $I^0=I\subset P_f$, define
\begin{equation*}
	I^1=\{(v,\alpha)\;|\;(v-1,\alpha)\in I\}.
\end{equation*}
Inductively define $I^n=(I^{n-1})^1$ for each positive integer $n$.
For each ideal $I\subset P_f$ define its height
\begin{equation*}
	h(I)=\min\{v\;|\;(v,\alpha)\in I\}.
\end{equation*}
As a convention, we say that the minimum of the empty set is $\infty$.
Let $\{\kappa(I)_n\}$ denote the sequence
\begin{equation*}
	\kappa(I)_n=h(I^n) \text{ for $n\in \Z_{\geq 0}$.}
\end{equation*}
\begin{theorem}
	The maps $\{h_n\}\mapsto I(\{h_n\})$ and $I\mapsto \kappa(I)$ define mutually inverse isomorphisms between the posets $H_f$ and $J(P_f)$.
\end{theorem}
\begin{proof}
  The rationale behind the construction of $\kappa(I)$ from the ideal $I$ is quite clear; for any element $a\in A$, the height of $a$ is $h(I(a))$ and $I(pa)=I^1$.
  It follows that $h(p^na)=h(I(a)^n)$ for each $n\geq 0$.
  Therefore $\kappa(I(a))$ is the height sequence of $a$.
	
  Thus, an element with ideal $I$ has Ulm sequence $\kappa(I)$.
  Since the ideal and the Ulm sequence give rise to isomorphisms from the poset of
  $G$-orbits in $A$ to $J(P_f)$ and $H_f$ respectively, $I\mapsto
  \kappa(I)$ is an isomorphism of posets $J(P_f)\to H_f$.

  Now suppose that
  \begin{equation*}
    \max I = \{ (v_1,\alpha_1),\dotsc,(v_k, \alpha_k)\}
  \end{equation*}
  where $\alpha_1<\dotsb <\alpha_k$.
  Since $\max I$ is an antichain in $P_\lambda$, it follows that
  $v_1<\dotsb<v_k$.
  In particular, $h(I) = v_1$.

  Likewise, for each $i\geq 0$,
  \begin{equation*}
    \max(I^i) = \{(v_j+i,\alpha_j)\;|\; v_j+i<\alpha_j\}.
  \end{equation*}
  Therefore,
  \begin{equation*}
    h(I^i) = \min\{v_j+i\;|\;v_j+i<\alpha_j\},
  \end{equation*}
  and
  \begin{equation*}
    h(I^{i-1}) = \min\{v_j+i-1\;|\;v_j+i-1<\alpha_j\}. 
  \end{equation*}
  Thus $h(I^i)>h(I^{i-1})+1$ if and only if there exists $j$ such that
  \begin{equation*}
    v_j+i-1<\alpha_j\leq v_j+i.
  \end{equation*}
  It follows that $\alpha_j = v_j+i$, and $h(I^{i-1})=v_j+i-1$.
  Thus
  \begin{equation*}
    (v_j,\alpha_j) = (h(I^{i-1})-i+1,h(I^{i-1})+1),
  \end{equation*}
  whence it follows that $I=I(\kappa(I))$.
\end{proof}

\end{document}